\documentclass[10pt, letterpaper, reqno]{amsart}

\usepackage[a4paper,left=30mm,right=30mm,top=25mm,bottom=25mm,marginpar=25mm]{geometry}

\usepackage{amsmath, amsthm, amsfonts, amssymb, amscd, mathtools, bm} 
\usepackage{xcolor} 
\usepackage[unicode=true]{hyperref} 
\usepackage[shortlabels]{enumitem} 
\usepackage{empheq}
\usepackage[T1]{fontenc}

\usepackage{tikz} %required for latex pictures (Loukas)
\usepackage{graphicx}  %required for latex pictures (Loukas)
\usepackage{caption, wrapfig} %required for \graphfigure () %%%

\theoremstyle{plain}
\newtheorem{theorem}{Theorem}[section]
\newtheorem{lemma}[theorem]{Lemma}

\newtheorem{proposition}[theorem]{Proposition}

\newtheorem{remark}[theorem]{Remark}

\theoremstyle{definition}
\newtheorem{definition}[theorem]{Definition}

\newcommand{\rd}{\mathrm{d}}

\begin{document}
  \numberwithin{equation}{section}
\title[]{A bilinear fractional integral operator \\ for Euler-Riesz systems}
\author[]{Nuno J. Alves}
\address[Nuno J. Alves]{
      University of Vienna, Faculty of Mathematics, Oskar-Morgenstern-Platz 1, 1090 Vienna, Austria.}
\email{nuno.januario.alves@univie.ac.at}

\author[]{Loukas Grafakos}
\address[Loukas Grafakos]{
      University of Missouri, Department of Mathematics, Columbia MO 65211, USA.}
\email{grafakosl@missouri.edu}

\author[]{Athanasios E. Tzavaras}
\address[Athanasios E. Tzavaras]{
      King Abdullah University of Science and Technology, CEMSE Division, Thuwal, Saudi Arabia, 23955-6900.}
      \email{athanasios.tzavaras@kaust.edu.sa}

\keywords{Marcinkiewicz interpolation, Euler-Riesz equations, compensated integrability}
\subjclass[2020]{42B20, 42B37, 35Q35} 

\maketitle
\thispagestyle{empty} 

\begin{abstract}
We establish a uniform estimate for a bilinear fractional integral operator via restricted weak-type endpoint estimates and   
Marcinkiewicz interpolation. This estimate is crucial in the integrability analysis of a tensor-valued bilinear fractional integral operator associated with Euler-Riesz systems modeling mean-field interactions induced by a singular kernel. The tensorial operator arises from a reformulation of the Euler-Riesz system that yields a gain in integrability for finite energy solutions through compensated integrability. Additionally, for smooth periodic solutions of the reformulated system, we derive a stability result.
\end{abstract}

\section{Introduction}
We consider the following Euler-Riesz system for $t \geq 0 $ and $x\in \mathbb{R}^d$ (with $d \in \mathbb{N}$):
%on $(t,x) \in (0,\infty) \times \mathbb{R}^d$ (with $d \in \mathbb{N}$):
\begin{equation} \label{ER0}
\begin{cases}
\partial_t \rho + \nabla \cdot (\rho u) = 0, \\
\partial_t (\rho u) + \nabla \cdot (\rho u \otimes u) +\nabla \rho^\gamma +  \rho \nabla K_\alpha \ast \rho = 0, 
\end{cases}
\end{equation}
where $\rho : [0,\infty) \times \mathbb{R}^d \to [0,\infty)$ denotes a density,  $u : [0,\infty) \times \mathbb{R}^d \to \mathbb{R}^d$ stands for the velocity and the exponent $\gamma$ is greater than $1$.
The kernel $K_\alpha$ is given by
\begin{equation} \label{kernelK}
K_\alpha(x) = \tfrac{1}{d - \alpha}|x|^{\alpha - d} 
\end{equation}
with $0 < \alpha < d$;  the term $\rho \nabla K_\alpha \ast \rho$ describes the nonlocal repelling interaction of particles. Smooth solutions $(\rho,u)$ of (\ref{ER0}) 
decaying sufficiently fast at infinity satisfy the conservation of energy and mass identities:
\begin{equation}\label{apbounds}
\frac{\mathrm{d}}{\mathrm{d}t} \int_{\mathbb{R}^d} \tfrac{1}{2}\rho |u|^2 + \tfrac{1}{\gamma - 1} \rho^\gamma + \tfrac{1}{2} \rho (K_\alpha \ast \rho) \, \mathrm{d}x = 0, \qquad \frac{\mathrm{d}}{\mathrm{d}t} \int_{\mathbb{R}^d} \rho \, \mathrm{d}x = 0 \, .
\end{equation}
This, in particular, yields an a priori estimate for weak solutions, which implies the regularity $ \rho \in L^\infty\big((0,\infty);L^1 \cap L^\gamma(\mathbb{R}^d)\big)$ for the density.
\par 
In this work, we  exploit an intriguing connection between harmonic analysis and the theory of  Euler-Riesz systems hinging on
the study of  a bilinear fractional integral operator.
The approach is based on a reformulation of the interaction term in divergence form, as seen in \eqref{divSalpha} below, in conjunction with 
uniform bounds for an associated bilinear fractional integral operator that are established here.
This reformulation is advantageous for three reasons: 
(i)  On the one hand, all the terms of the equations are written in divergence form, 
allowing the derivatives to be absorbed by the test functions in a weak formulation.
(ii) The harmonic analysis estimates lead to integrability properties of the nonlocal interaction term.
(iii) Finally, for finite energy solutions, it provides a higher integrability estimate for the density in space-time, achieved by applying the compensated integrability theory for divergence-free positive symmetric tensors \cite{serre2018divergence, serre2019compensated, serre2023mixed} to the setting of Euler-Riesz systems. 
 \par
To illustrate, note that the only term of (\ref{ER0}) that is not in divergence form is the interaction term $\rho \nabla K_\alpha \ast \rho$. Inspired by a calculation in \cite{serre2019compensated} and exploiting the symmetry of the kernel $K_\alpha$, one reaches the  identity
\begin{equation} \label{divSalpha}
 \rho \nabla K_\alpha \ast \rho = \nabla \cdot S_\alpha(\rho)  
\end{equation}
where $S_\alpha(\rho)$ is a tensor defined by 
\begin{equation} \label{tensorS}
S_\alpha(\rho)(t,x) =  \tfrac{1}{2} \int_0^1  \int_{\mathbb{R}^d} \rho(t, x + (\theta -1)y) \, \rho(t, x + \theta y) \, |y|^{\alpha -d-2 } \, y\otimes y \, \rd y \, \rd \theta.
\end{equation}
Identity (\ref{divSalpha}) is derived in Appendix \ref{appendix1} and yields a reformulation of (\ref{ER0}) in which the equations are expressed as a divergence-free condition of a tensor that fits into the compensated integrability framework of \cite{serre2018divergence}. 
This reformulation leads, in turn,  to a higher integrability estimate for finite energy solutions  thereby improving on the integrability 
provided by the energy identity; see Theorem \ref{higherintegrability}. 
\par 
To analyze $S_\alpha(\rho)$, we consider a bilinear fractional integral operator $I_\alpha^\theta$ defined for  nonnegative measurable functions $f$ and $g$ on $\mathbb{R}^d$ by
\begin{equation} \label{integraloperator}
I_\alpha^\theta(f,g)(x) = \int_{\mathbb{R}^d} f( x + (\theta -1)y) \, g (x + \theta y) \, |y|^{\alpha -d } \, \mathrm{d}y
\end{equation} 
with $0 < \alpha < d$ and $0 \leq  \theta \leq 1$. 
The main result of this work provides a uniform bound in $\theta$ for $I_\alpha^\theta$ with assumptions similar in style to the classical Hardy-Littlewood-Sobolev (HLS) inequality; see Theorem \ref{uniformestimate}.   
\par
From the natural integrability of $\rho$ induced by the energy identity, one may deduce using the classical HLS inequality that the term 
$\rho \nabla K_\alpha \ast \rho$ belongs to $L^1$ in space whenever $1 < \alpha < d$.  By contrast, when employing the 
formulation \eqref{divSalpha} via the tensor $S_\alpha(\rho)$, one improves the  range to $0 < \alpha < d$.
This observation underlines the importance of the reformulation of (\ref{ER0}) through identity (\ref{divSalpha}). 
\par
The paper is organized as follows. In Section \ref{section_results} we state the main theorem of this work and describe the associated results. In Section \ref{section_CI} we explain how the theory of compensated integrability leads to a higher integrability estimate for finite energy solutions of the Euler-Riesz system. Section \ref{section_HA} contains the proof of the main theorem and its corollary, Proposition \ref{Jalphaestimate}, which yields an integrability result for the tensor given by (\ref{tensorS}). Finally, in Section \ref{section_ER}, we establish a stability result for smooth periodic solutions of the reformulated Euler-Riesz system via the relative energy method.

%\newpage
\section{Description of results} \label{section_results}
The main theorem of this work provides a uniform estimate for the bilinear fractional operator $I_\alpha^\theta$ given by (\ref{integraloperator}); see Theorem \ref{uniformestimate} below. \par 
Fractional integral operators have been of great importance  in harmonic analysis for several decades; however, in recent years, their bilinear analogues have also attracted research attention. In particular, an operator $B_\alpha$, with $0 < \alpha < d$, acting on nonnegative measurable functions of $\mathbb{R}^d$ as
\[B_\alpha(f,g) = \int_{\mathbb{R}^d} f( x  - y) \, g (x + y) \, |y|^{\alpha -d } \, \mathrm{d}y \]
was first considered in \cite{grafakos1992multilinear} and later in \cite{kenig1999multilinear, grafakos2001some}, in which optimal boundedness properties between Lebesgue spaces were established. It has subsequently been studied extensively by several authors; we refer to \cite{ding2002rough, moen2014new, li2016two, hoang2018weighted, hatano2019note, furuya2020weighted, he2021bilinear} for estimates concerning $B_\alpha$ (and related versions) on a variety of spaces.
While the operator $I_\alpha^\theta$ is quite similar to $B_\alpha$ when the dependence on the parameter $\theta$ is ignored, its study becomes more intricate when seeking estimates that are uniform in the auxiliary parameter $\theta$.
 \par 
These types of operators have sparked significant interest primarily due to the singular nature of their integrands, but also due to their proximity to Hilbert transforms. Notable examples include the linear fractional integral operator (also known as the Riesz potential) and the linear Hilbert transform. In our case, the bilinear operator $I_\alpha^\theta$ is related to a bilinear Hilbert transform  $H^\theta$, given by
\[H^\theta(f,g)(x) = \text{p.v.} \int_\mathbb{R} f(x + (\theta - 1)t) \, g(x + \theta t) \frac{\mathrm{d}t}{t}. \]
Uniform bounds in $\theta$ for this transform can be deduced by direct application of the results obtained in \cite{grafakos2004uniform, liu2006uniform}. Other boundedness results for similar bilinear Hilbert transforms can be found in \cite{lacey1997p,lacey1999calderon}. \par 
\subsection{Main result}
\begin{theorem} \label{uniformestimate} 
Let $d\in \mathbb{N}$ be the dimension,  $0<\alpha<d$, and $p,q,r$ be integrability exponents satisfying 
\[1 < p,q < \dfrac{d}{\alpha} \quad \text{and} \quad \dfrac{1}{p}+ \dfrac{1}{q} = \dfrac{1}{r} + \dfrac{\alpha}{d}. \] 
Then there is a constant   $C = C(\alpha, d, p ,q) > 0$ independent of $\theta$ such that  for all $f \in L^p(\mathbb{R}^d)$ and $g \in L^q(\mathbb{R}^d)$ we have
 \begin{equation}\label{uestimate}
 \|  I_\alpha^\theta(f,g) \|_{L^r(\mathbb{R}^d)} \leq C \, \| f\|_{L^p(\mathbb{R}^d)} \, \|g \|_{L^q(\mathbb{R}^d)}.
 \end{equation}
\end{theorem}
\par  
Whenever  $(1/p,1/q)$ lies in the interior of the square with vertices 
$(\alpha/d, \alpha/d)$, 
$(\alpha/d, 1)$, 
$(1,\alpha/d)$, and 
$(1,1)$, then $I_\alpha^\theta$ 
is bounded from $L^p(\mathbb R^d)
\times  L^q(\mathbb R^d)$ to $L^r(\mathbb R^d)$
{\it uniformly  in $\theta$} when $1/p+1/q=1/r+\alpha/d$.
If one ignores the uniform bounds in $\theta$, then 
$I_\alpha^\theta$ is bounded from $L^p(\mathbb R^d)
\times  L^q(\mathbb R^d)$ to $L^r(\mathbb R^d)$ when the pair $(1/p,1/q)$ lies in the interior of the pentagon with vertices
$(0,\alpha/d )$,
$(0,1)$,
$(1,1)$
$(1,0)$, and 
$(\alpha/d,0)$.
See Figure~\ref{Indices}. 

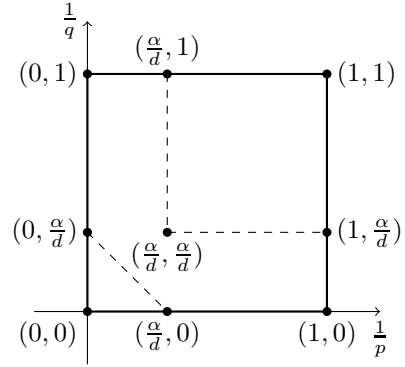
\begin{wrapfigure}{R}{.42\textwidth} 
\centering 
\begin{tikzpicture}[scale=.35, point/.style={fill, circle, inner sep=1.2pt}]
\draw[->] (0,-2) -- (0,11) node[left] {$\frac 1 q$}; 
\draw[->] (-2,0) -- (11,0) node[below] {$\frac 1 p$}; 

\draw[thick] (0,0)  -- 
	(9,0) node[point]{} node[below]{$(1,0)$} -- 
	(9,9) node[point]{} node[right]{$(1,1)$} -- 
	(0,9) node[point]{} node[left]{$(0,1)$} -- 
	(0,0) node[point]{} node[below left]{$(0,0)$}; 

\draw[dashed] (0,3) node[point]{} node[left]{$(0, \frac \alpha d)$} -- 
	(3,0) node[point]{} node[below]{$(\frac \alpha d,0)$}; 

\draw[dashed] (3,9) node[point]{} node[above]{$(\frac \alpha d, 1)$} -- 
	(3,3) node[point]{} node[below]{$(\frac \alpha d, \frac \alpha d)$} -- 
	(9,3) node[point]{} node[right]{$(1, \frac \alpha d)$}; 
\end{tikzpicture}
\caption{Region of boundedness}\label{Indices}
\end{wrapfigure}

 The proof of Theorem~\ref{uniformestimate} relies on a bilinear version of the Marcinkiewicz interpolation method, where from a finite set of restricted weak-type estimates, one deduces strong-type estimates; see Proposition \ref{Minterpolation}. A more general version of this method, for multilinear operators, was established in \cite{grafakos2001some, grafakos2012multilinear}. In particular, in \cite{grafakos2001some}, this method is deduced as a corollary of a Boyd interpolation theorem in a framework of quasi-normed rearrangement-invariant spaces. 
\par
Finally, we would like to point out that the largest possible region in which uniform estimates hold for $I_\alpha^\theta$ is, in fact, the open square with vertices $(\alpha/d, \alpha/d)$, $(\alpha/d, 1)$, $(1,\alpha/d)$, and $(1,1)$.   
In fact, by interpolation, it suffices to verify that 
uniform bounds fail  on  the boundary of this square.
To verify this assertion, let us assume that a uniform bound  
\[
\sup_{0<\theta< 1} \| I_\alpha^\theta(f, g) \|_{L^r(\mathbb R^d)} 
\le C\, \| f  \|_{L^p(\mathbb R^d)} 
 \| g \|_{L^{\frac{d}{\alpha}} (\mathbb R^d)} 
\]
holds on the horizontal dotted line, that is, when 
$1/p+1/q = 1/r+\alpha/d$ and $q=d/\alpha$, for some positive constant $C = C(\alpha,d)$. In this case, we must have $p=r$. By Fatou's lemma, it follows that 
\begin{equation*}
\|\liminf_{\theta\to 1}  I_\alpha^\theta(f, g) \|_{L^p(\mathbb R^d)} 
\le C \, \| f  \|_{L^p(\mathbb R^d)} 
 \| g \|_{L^{\frac{d}{\alpha}} (\mathbb R^d)} 
\end{equation*}
hence, for all Schwartz functions $f$ and $g$ we must have 
\begin{equation}\label{new1}
\| f I_\alpha (g ) \|_{L^p(\mathbb R^d)} 
\le C \, \| f  \|_{L^p(\mathbb R^d)} 
 \| g \|_{L^{\frac{d}{\alpha}} (\mathbb R^d)}
\end{equation}
where $I_\alpha$ is the fractional integral operator
\[
I_\alpha (g)(x) = \int_{\mathbb{R}^d} g( x +y)   \, |y|^{\alpha -d } \, \rd y  
= \int_{\mathbb{R}^d} g( x -y)   \, |y|^{\alpha -d } \, \rd y.
\]
Now inserting $f(x) = f_{\epsilon,x_0} (x) = (1/\epsilon)^{\frac d2}
e^{-\frac \pi \epsilon |x-x_0|^2} $ in \eqref{new1} and letting $
\epsilon\to 0$, we obtain 
\[
|I_\alpha(g)(x_0)| \le C \,
 \| g \|_{L^{\frac{d}{\alpha}} (\mathbb R^d)} 
\]
for all $x_0\in \mathbb R^d$. 
This would imply that $I_\alpha $ maps $L^{\frac{d}{\alpha}} (\mathbb R^d)$ to $L^\infty(\mathbb R^d)$, a fact known to be false; see
\cite[Example 5.1.4]{grafakos2024fundamentals}. An analogous argument (letting 
$\theta \to 0$) indicates that a uniform bound also cannot  hold on the vertical dotted line of Figure~\ref{Indices}.

\subsection{Connections with the HLS inequality} \label{section_HLS}
~\par
Recall the HLS inequality \cite{lieb2001analysis}:
\begin{proposition} \label{HLSprop}
Let $p,q > 1$ and $0 < \alpha < d$ satisfy 
\[\frac{1}{p} + \frac{1}{q} = 1 + \frac{\alpha}{d}. \]
 If $f \in L^p(\mathbb{R}^d)$ and $g \in L^q(\mathbb{R}^d)$, then 
\begin{equation} \label{HLS}
\left| \int_{\mathbb{R}^d} \int_{\mathbb{R}^d} f(x) |x - y|^{\alpha - d} g(y) \, \rd x \, \rd y \right| \leq C \, \|f \|_{L^p(\mathbb{R}^d)} \, \|g \|_{L^q(\mathbb{R}^d)}
\end{equation}
for some $C = C(\alpha, d, p) > 0$.
\end{proposition}
Note that the assumptions on the integrability exponents $p$ and $q$ in Theorem \ref{uniformestimate} with $r = 1$ and Proposition \ref{HLSprop} are exactly the same. Indeed, the assumptions of Proposition \ref{HLSprop} imply that $p,q < d/\alpha$. To check this fact, suppose, without loss of generality, that $p > 1$ and $q \geq d/\alpha$. Then $1/p + 1/q < 1 + \alpha/d$,
which is a contradiction. \par 
Furthermore, the HLS inequality can be used to infer the $L^1$ boundedness of the operator $I_\alpha^\theta$, since for nonnegative measurable functions $f$ and $g$, appropriate changes of variables yield 
\[\|I_\alpha^\theta(f,g) \|_{L^1{(\mathbb{R}^d})} = \int_{\mathbb{R}^d} \int_{\mathbb{R}^d} f(x) |x - y|^{\alpha - d} g(y) \, \rd x \, \rd y. \]
Thus, the uniform estimate (\ref{uestimate}) is particularly important if $r \neq 1$.

\subsection{A tensorial bilinear fractional integral operator}
~\par
Consider a bilinear version of the tensor $S_\alpha(\rho)$, that is, define a tensorial bilinear fractional integral operator $J_\alpha$ for nonnegative measurable functions $f$ and $g$ on $\mathbb{R}^d$ by 
\begin{equation} \label{Jalpha}
J_\alpha(f,g)(x) = \int_0^1 \int_{\mathbb{R}^d} f( x + (\theta -1)y) \, g (x + \theta y) \, |y|^{\alpha -d -2 } \, y \otimes y \, \mathrm{d}y \, \mathrm{d}\theta.
\end{equation}
Note that $S_\alpha(\rho) = \tfrac{1}{2}J_\alpha(\rho,\rho)$. Additionally, identity $(\ref{divSalpha})$ can be written in terms of $J_\alpha$ as 
\begin{equation} \label{tensor_relation}
f \nabla K_\alpha \ast f = \nabla \cdot \left( \tfrac{1}{2} J_\alpha(f,f) \right). 
\end{equation}
As a consequence of Theorem \ref{uniformestimate}, for the operator $J_\alpha$ we obtain the following result:
\begin{proposition} \label{THLS}
Let   $0<\alpha<d$  and $p,q,r$ be integrability exponents satisfying 
 \[1 < p,q < \dfrac{d}{\alpha}, \quad  r \geq 1, \quad \text{and} \quad \dfrac{1}{p}+ \dfrac{1}{q} = \dfrac{1}{r} + \dfrac{\alpha}{d}. \] Then for all 
  $f \in L^p(\mathbb{R}^d)$ and $g \in L^q(\mathbb{R}^d)$ we have 
\begin{equation} \label{Jalphaestimate}
\|J_\alpha (f,g) \|_{L^r(\mathbb{R}^d)} \leq C \, \| f\|_{L^p(\mathbb{R}^d)} \, \|g \|_{L^q(\mathbb{R}^d)}
\end{equation} 
for some $C = C(\alpha, d,p,q) > 0$. 
\end{proposition}
The previous proposition leads to an integrability result for $S_\alpha(\rho)$. Indeed, if 
\[\rho \in L^1 (\mathbb{R}^d) \cap  L^\gamma(\mathbb{R}^d)\]
-- as implied by the a priori bounds \eqref{apbounds} -- and $ 1 < 2dr/(d+\alpha r) \leq \gamma$, for some $r \geq 1$, then
$S_\alpha(\rho) \in L^r (\mathbb{R}^d)$ and there exists a constant $C = C(\alpha, d, r) > 0$ such that 
\begin{equation} \label{Sestimate}
\|S_\alpha (\rho) \|_{L^r(\mathbb{R}^d)} \leq C \, \|\rho \|_{L^p (\mathbb{R}^d)}^2 
\end{equation} 
where $p = 2dr/(d+\alpha r)$. Note that by interpolation, the right hand side of (\ref{Sestimate}) is controlled by the norm $\| \rho \|_{L^1(\mathbb{R}^d)} + \| \rho \|_{L^\gamma(\mathbb{R}^d)}$.

\subsection{Reformulation of the Euler-Riesz system} \label{section_motivation}
~\par
Consider the Euler-Riesz system (\ref{ER0}) supplemented with initial data $\rho_0$ and $u_0$. This system comprises a continuity equation for the conservation of mass and a second equation that ensures the conservation of momentum. These equations govern the dynamics of a compressible fluid with density $\rho$ and linear velocity $u$, subject to pressure and interaction forces. The pressure function is given by $ p(\rho)=\rho^\gamma,$ with $\gamma > 1$ being the adiabatic exponent, and the interaction forces are modelled through the kernel $K_\alpha$ given by (\ref{kernelK}). For $d \geq 3$ and $\alpha = 2$, we recover the Euler-Poisson equations, as in that case, the interaction kernel $K_2$ is the Newtonian kernel. For existence theories on Euler-Riesz systems, we refer to \cite{choi2022well, danchin2022global}.
\par 
As observed in the introduction, smooth solutions of \eqref{ER0} satisfy a priori bounds of conservation of energy and mass. Given the adiabatic exponent, it is reasonable to consider solutions such that $\rho$ belongs to $L^1 \cap L^\gamma$ in space. A natural question is whether this integrability can be improved by exploiting the structure of the equations. For finite energy solutions this can be accomplished by compensated integrability. Specifically, for a finite energy solution $(\rho,u)$ of (\ref{ER0}), one can prove that for each $T > 0$
\begin{equation} \label{integrabilityimprovement}
\text{if} \quad  \rho \in L^\infty\big(0,T; L^1(\mathbb{R}^d) \cap L^\gamma(\mathbb{R}^d) \big) \quad \text{then} \quad \rho \in L^{\gamma + \frac{1}{d}}\big( (0,T) \times \mathbb{R}^d \big).
\end{equation}
\par
The first step towards (\ref{integrabilityimprovement}) is to rewrite system (\ref{ER0}) as a space-time divergence-free condition for an appropriate tensor. This is made possible through identity (\ref{divSalpha}). This reformulates system (\ref{ER0}) into a divergence-free positive symmetric tensor form,
 fitting in the compensated integrability theory of \cite{serre2018divergence, serre2019compensated, serre2023mixed}, thereby yielding the integrability improvement (\ref{integrabilityimprovement}); see Section \ref{section_CI}. We refer to \cite{guerra2024compensation} for an extension of this theory, and to \cite{lefloch2007finite} where a higher integrability estimate is obtained for one-dimensional finite energy solutions of an isentropic Euler system using a different methodology.
\par 
Next, we explore a possible weak formulation for the Euler-Riesz system (\ref{ER0}). For the continuity equation take
\begin{equation*} 
\int_0^\infty \int_{\mathbb{R}^d} \rho  \partial_t \varphi + \rho u \cdot \nabla  \varphi\, \mathrm{d}x  \, \mathrm{d}t + \int_{\mathbb{R}^d} \rho_0 \varphi_0 \, \mathrm{d}x = 0
\end{equation*}
where $\varphi \in C^1_c([0,\infty) \times \mathbb{R}^d)$ is a test function with $\varphi|_{t=0} = \varphi_0$. For the momentum equation we have two options, according to identity (\ref{divSalpha}). Let $\xi \in C^1_c([0,\infty) \times \mathbb{R}^d; \mathbb{R}^d)$ be a test function with $\xi|_{t=0} = \xi_0$. Using the left-hand side of (\ref{divSalpha}), we get
\begin{equation} \label{weakform21}
\int_0^\infty \int_{\mathbb{R}^d} \rho u \cdot \partial_t \xi + (\rho u \otimes u +\rho^\gamma I_d) : \nabla \xi -  \rho \nabla (K_\alpha \ast \rho) \cdot \xi \, \mathrm{d}x \, \mathrm{d}t + \int_{\mathbb{R}^d} \rho_0 u_0 \cdot \xi_0 \, \mathrm{d}x = 0,
\end{equation}
whereas, using the right-hand side of (\ref{divSalpha}), we have
\begin{equation} \label{weakform22}
\int_0^\infty \int_{\mathbb{R}^d} \rho u \cdot \partial_t \xi + \big(\rho u \otimes u +\rho^\gamma I_d +   S_\alpha(\rho) \big) : \nabla \xi \, \mathrm{d}x \, \mathrm{d}t + \int_{\mathbb{R}^d} \rho_0 u_0 \cdot \xi_0 \, \mathrm{d}x = 0,
\end{equation}
where $I_d$ is the $d\times d$ identity matrix, and for two square matrices $A=(a_{ij})$ and $B=(b_{ij})$, $A : B = \sum_{i,j} a_{ij} b_{ij}$. \par
Assume that $\rho \in L^1 \cap L^\gamma(\mathbb{R}^d)$.  Using the  HLS inequality we obtain:
\begin{enumerate}[(i)]
\item If $1 < \alpha < d$ and $\gamma \geq q = 2d/(d+\alpha -1)$, then $\rho \nabla K_\alpha \ast \rho \in L^1(\mathbb{R}^d)$ since \[\|\rho \nabla K_\alpha \ast \rho \|_{L^1(\mathbb{R}^d)} \leq C(\alpha, d) \, \|\rho \|_{L^{q}(\mathbb{R}^d)}^2.\]
\item If $0 < \alpha < d$ and $\gamma \geq p = 2d/(d+\alpha)$, then $S_\alpha(\rho) \in L^1(\mathbb{R}^d)$ since \[\|S_\alpha(\rho) \|_{L^1(\mathbb{R}^d)} \leq C(\alpha, d) \, \|\rho \|_{L^{p}(\mathbb{R}^d)}^2.\]
\end{enumerate}
The second formulation is preferable as it is well-defined for a larger range of the parameters $\alpha $ and $\gamma$.\par 

\section{Compensated integrability}\label{section_CI}
In this section we provide a proof for (\ref{integrabilityimprovement}), which is one of the reasons for having considered the tensor $S_\alpha(\rho)$ and subsequently the bilinear fractional integral operator $I_\alpha^\theta$. 
\subsection{A divergence-free positive symmetric tensor} \label{section_CI_DPT}
~ \par 
First, we write system (\ref{ER0}) as a space-time divergence-free condition for an appropriate tensor. Thanks to (\ref{divSalpha}), system $(\ref{ER0})$ can be reformulated into:
\begin{equation} \label{ERDPT}
\nabla_{t,x} \cdot A_\alpha(\rho,u) = 0
\end{equation}
where the $(1+d)$-tensor $A_\alpha(\rho,u)$ is given by 
\begin{equation} \label{tensorA}
A_\alpha(\rho,u) = 
\begin{bmatrix}
\rho & (\rho u)^\top  \\
\rho u & \rho u\otimes u + p(\rho)I_d + S_\alpha(\rho)  \\
\end{bmatrix}.
\end{equation}
\par 
Next, we deduce some basic properties of the tensor $S_\alpha(\rho)$ given by (\ref{tensorS}) that are relevant for the subsequent analysis. 
\begin{proposition} \label{detbounds}
The tensor $S_\alpha(\rho)$ is symmetric, positive semi-definite and 
\begin{equation*}
\det \big(p(\rho)I_d + S_\alpha(\rho) \big) \geq \begin{cases}
p(\rho)^d, \\
\det S_\alpha(\rho).
\end{cases} 
\end{equation*} 
\end{proposition}
\begin{proof}
It is clear that $S_\alpha(\rho)$ is symmetric since $y \otimes y$ is symmetric. Moreover, given a vector $v=v(x)$,
\[v^\top S_\alpha(\rho) v =  \tfrac{1}{2} \int_{\mathbb{R}^d} \int_0^1 \rho(x+(\theta - 1)y) \rho(x+\theta y) |y|^{\alpha-d - 2} (y \cdot v)^2 \, \mathrm{d}\theta \, \mathrm{d}y \geq 0\]
hence $S_\alpha(\rho)$ is positive semi-definite. Therefore, there exist nonnegative eigenvalues $\lambda_1, \ldots, \lambda_d,$ and respective eigenvectors $v_1, \ldots, v_d,$ that is, $S_\alpha(\rho) v_i = \lambda_i v_i$. Then $p(\rho) + \lambda_i$ is an eigenvalue of $p(\rho)I_d + S_\alpha(\rho)$, since $\big(p(\rho) I_d + S_\alpha(\rho) \big) v_i = \big(p(\rho) + \lambda_i \big)v_i$. Hence, given that $\lambda_i \geq 0,$
\begin{equation*}
\det \big(p(\rho)I_d + S_\alpha(\rho) \big) = \prod_{i=1}^d \big(p(\rho) + \lambda_i \big) \geq \begin{dcases}
\prod_{i=1}^d p(\rho) = p(\rho)^d, \\
\prod_{i=1}^d \lambda_i = \det S_\alpha(\rho).
\end{dcases} 
\end{equation*}
\end{proof}
\par 
It follows that the tensor $A_\alpha(\rho,u)$ is symmetric and positive semi-definite. To check the latter, let $w = (w_0, \tilde{w}),$ with $w_0$ being a scalar and $\tilde{w}$ a $d$-dimensional vector, and note that 
\begin{align*}
w^\top A w &= \begin{bmatrix}
w_0 & \tilde{w}^\top
\end{bmatrix}
\begin{bmatrix}
\rho & (\rho u)^\top  \\
\rho u & \rho u\otimes u + p(\rho)I + S_\alpha(\rho)  \\
\end{bmatrix}
\begin{bmatrix}
w_0 \\
\tilde{w}
\end{bmatrix} \\
& = \rho(w_0 + u \cdot \tilde{w})^2 + p(\rho)|\tilde{w}|^2 + \tilde{w}^\top S_\alpha(\rho) \tilde{w} \\
& \geq 0 
\end{align*} 
where in the last step we used the fact that $S_\alpha(\rho)$ is positive semi-definite. \par 
Consequently $A_\alpha(\rho,u)$ is a divergence-free positive symmetric tensor. 

\subsection{Higher integrability for finite energy solutions} \label{section_CI_HI}
~\par
Assume that $(\rho,u)$ is a solution of (\ref{ERDPT}) with finite mass and energy and such that $A_\alpha(\rho,u)$ belongs to $L^1\big((0,T) \times \mathbb{R}^d \big) \cap L_{loc}^{1+\frac1d} \big( (0,T) \times \mathbb{R}^d \big)$ for each $T > 0$. Note that by the conservation of mass and energy, it suffices to prescribe initial data $(\rho_0, u_0)$ with finite mass and energy. We apply \cite[Theorem 2.3]{serre2018divergence} to the tensor $A_\alpha(\rho,u)$, along the same lines of the proof of \cite[Theorem 3.1]{serre2018divergence}. \par Set \[\Sigma = (0,T) \times B_R, \quad B_R = \{x \in \mathbb{R}^d \ | \ |x| < R \}, \] 
\[\partial \Sigma = \big(\{ 0 \} \times B_R \big) \cup \big( (0,T)\times \partial B_R \big) \cup \big( \{ T \} \times B_R \big). \]
 The following estimate holds:
\begin{equation} \label{initialestimate}
\int_0^T \int_{B_R} \big(\det A_\alpha(\rho,u) \big)^{\frac{1}{d}} \ \mathrm{d}x \, \mathrm{d}t \leq c_d \| A_\alpha(\rho,u) \nu \|_{L^1(\partial \Sigma)}^{1+\frac{1}{d}}
\end{equation}
where $\nu$ is the outward normal vector to the boundary of $\Sigma,$ given by
\[
\nu = \begin{cases}
(-1, 0_d) & \text{on} \ \{0\} \times B_R, \\
z = (0, x / |x|) & \text{on} \ (0,T) \times B_R, \\
(1, 0_d) & \text{on} \ \{T\} \times B_R.
\end{cases}
\]
Hence
\[ \| A_\alpha(\rho,u) \nu \|_{L^1(\partial \Sigma)} = \int_{B_R} |(\rho, \rho u)|_{t=0} + |(\rho, \rho u)|_{t=T} \, \mathrm{d}x + \psi(R)\]
where \[\psi(R) = \int_0^T \int_{\partial B_R} |A_\alpha(\rho,u)z| \, \mathrm{d}x\,\mathrm{d}t  .\]
Since $A_\alpha(\rho,u)$ is integrable it follows that $\psi \in L^1(0, \infty).$ Indeed,
\begin{align*}
\int_0^{\infty} |\psi(R)| \, \mathrm{d}R & \leq \int_0^{\infty} \int_0^T \int_{\partial B_R} |A_\alpha(\rho,u)| \, \mathrm{d}x \, \mathrm{d}t \, \mathrm{d}R \\
& =  \int_0^{\infty}    \frac{\rd}{\mathrm{d}R} \int_0^T \int_{B_R} |A_\alpha(\rho,u)| \, \mathrm{d}x \, \mathrm{d}t \, \mathrm{d}R \\ 
& = \|A_\alpha(\rho,u) \|_{L^1 ( (0,T) \times \mathbb{R}^d )}.
\end{align*}
Therefore, there exists a sequence $R_n \to \infty$ such that $\psi(R_n) \to 0.$ Considering this limit in (\ref{initialestimate}), and using the conservation of mass and momentum, gives 
\begin{align*}
\int_0^T \int_{\mathbb{R}^d} \big(\det A_\alpha(\rho,u)\big)^{\frac{1}{d}} \, \mathrm{d}x \, \mathrm{d}t & \leq c_d \Big( \int_{\mathbb{R}^d} |(\rho, \rho u)|_{t=0} + |(\rho, \rho u)|_{t=T} \, \mathrm{d}x \Big)^{1+\frac{1}{d}} \\
& = 2c_d \Big( \int_{\mathbb{R}^d} \sqrt{\rho_0^2 + \rho_0^2 |u_0|^2} \, \mathrm{d}x \Big)^{1+\frac{1}{d}} \\
& \leq 2c_d \Big( \int_{\mathbb{R}^d} \rho_0 + \rho_0 |u_0| \, \mathrm{d}x \Big)^{1+\frac{1}{d}}  
\end{align*}
where in the last inequality we used that $\sqrt{a^2 + b^2} \leq a+b$ for $a,b \geq 0$. \par 
Now, using Proposition \ref{detbounds},
\begin{align*}
\det A_\alpha(\rho,u) & =  \rho \det \left(\rho u \otimes u + p(\rho) I + S_\alpha(\rho) - \rho u \frac{1}{\rho} \rho u^\top  \right) \\
& = \rho \det \big( p(\rho) I_d + S_\alpha(\rho) \big) \\
& \geq \rho p(\rho)^d  .
\end{align*}
Consequently,
\begin{align*}
\int_0^T \int_{\mathbb{R}^d} \rho^{\frac{1}{d}} p(\rho) \, \mathrm{d}x \, \mathrm{d}t & = \int_0^T \int_{\mathbb{R}^d} \left(\rho p(\rho)^d \right)^{\frac{1}{d}} \, \mathrm{d}x \, \mathrm{d}t \\
& \leq \int_0^T \int_{\mathbb{R}^d} \left(\det A_\alpha(\rho,u)\right)^{\frac{1}{d}} \, \mathrm{d}x \, \mathrm{d}t \\
& \leq 2c_d \left( \int_{\mathbb{R}^d} \tfrac{3}{2}\rho_0 + \tfrac{1}{2}\rho_0 |u_0|^2 \, \mathrm{d}x \right)^{1+\frac{1}{d}} \\
& \leq 2c_d \left( \tfrac{3}{2} \int_{\mathbb{R}^d} \rho_0 \, \mathrm{d}x+ \int_{\mathbb{R}^d} \tfrac{1}{2}\rho_0 |u_0|^2+h(\rho_0)+ \tfrac{1}{2} \rho_0 K \ast \rho_0  \, \mathrm{d}x \right)^{1+\frac{1}{d}} \\
& < \infty
\end{align*}
which establishes the desired higher integrability estimate given that $p(\rho) = \rho^\gamma$. 

Letting $T \to \infty$ we obtain:
%\[\rho \in L^{\gamma + \frac{1}{d}} \big( (0, \infty) \times \mathbb{R}^d \big). \]
\begin{theorem}\label{higherintegrability}
Solutions of the Euler-Riesz system \eqref{ER0} with $0 < \alpha < d$ and repelling potentials satisfy the
a priori estimate
\begin{equation*}
\rho \in L^{\gamma + \frac{1}{d}} \big( (0, \infty) \times \mathbb{R}^d \big). 
\end{equation*}
\end{theorem}

\begin{remark}
Note that the special case $\alpha = 2$ with $d \ge 3$ corresponds to the Euler-Poisson system
\begin{equation} \label{EP}
\begin{aligned}
\partial_t \rho + \nabla \cdot (\rho u) &= 0, \\
\partial_t (\rho u) + \nabla \cdot (\rho u \otimes u) &= - \nabla \rho^\gamma  -  \rho \nabla \phi,
\\
- \Delta \phi &= \rho,
\end{aligned}
\end{equation}
commonly used in models for electrically charged fluids. \par In this case, as the potential is given as the solution of Poisson's equation, one could also write
\[\rho \nabla \phi =  \nabla \cdot \left(\tfrac{1}{2}|\nabla \phi|^2I_d - \nabla \phi \otimes \nabla \phi \right) \]
however, the tensor being applied by the divergence on the right-hand side is not positive semi-definite, and therefore it does not fit into the theory of compensated integrability.
\end{remark}

\section{Bilinear harmonic analysis}\label{section_HA}
The aim of this section is to prove Theorem \ref{uniformestimate}. Since all the operators involved are positive, we assume that all the considered functions are nonnegative. Given that all the Lebesgue spaces in this section are over $\mathbb{R}^d$, we shorten the notation of $L^p(\mathbb{R}^d)$ to $L^p$.

\subsection{An auxiliary operator \texorpdfstring{$I^\theta$}{\unichar{"1D43C}\unichar{"1DBF}}} \label{section_HA_auxiliary}
$ $\par
For $0 \leq \theta \leq 1$, let $I^\theta$ be the bilinear operator defined for (nonnegative) measurable functions $f$ and $g$ on $\mathbb R^d$ by
\begin{equation*} \label{Itheta}
I^\theta(f,g)(x) = \int_{|y| \leq 1} f(x + (\theta -1)y) \, g(x + \theta y) \, \mathrm{d}y.
\end{equation*}

\begin{lemma}
The operator $I^\theta$ maps $L^1 \times L^1$ to $L^{\frac{1}{2}}$ uniformly in $\theta$. Precisely:
\begin{equation} \label{auxI}
\| I^\theta(f,g) \|_{L^{\frac{1}{2}}} \leq C \, \|f \|_{L^1} \, \|g \|_{L^1}
\end{equation}
with $C = 3^d5^{2d}$.
\end{lemma}
\begin{proof}
We first prove (\ref{auxI}) with $C = 3^d$ for integrable functions $f$ and $g$ supported in cubes with sides of length one parallel to the axes. Let $Q_0 = [0,1]^d$ and for each $k \in \mathbb{Z}^d$, let $Q_k = k + Q_0$ denote the cube with side length one whose sides are parallel to the axes and whose lower left corner is $k$. For $k = (k_1, \ldots, k_d)$ and $l = (l_1, \ldots, l_d)$ in $\mathbb{Z}^d$, assume that $f$ is supported in $Q_k$ and that $g$ is supported in $Q_l$. Under these conditions, we claim that $I^\theta(f,g)$ is supported in a cube $Q$ of side length $3$. Indeed, for each $i= 1, \ldots, d$, the inequalities
\begin{equation*}
k_i \leq x_i + (\theta -1)y_i \leq k_i + 1, \qquad
l_i \leq x_i + \theta y_i \leq l_i + 1,
\end{equation*}
together with $|y| \leq 1$ and $0 \leq \theta \leq 1$ imply that
\begin{equation*}
k_i - 1\leq x_i  \leq k_i + 2, \qquad
l_i - 1\leq x_i \leq l_i + 2,
\end{equation*}
which establishes the claim. Thus, for these $f$ and $g$, the Cauchy-Schwarz inequality gives
\begin{equation*}
\begin{split}
\|I^\theta(f,g) \|_{L^{\frac{1}{2}}} & = \left(\int_{\mathbb{R}^d} \chi_Q |I^\theta(f,g)|^{\frac{1}{2}} \, \mathrm{d}x \right)^2 \\
& \leq 3^d \int_{\mathbb{R}^d} \int_{\mathbb{R}^d} f(x + (\theta -1)y) g(x + \theta y) \, \mathrm{d}y \, \mathrm{d}x \\
& \leq 3^d \int_{\mathbb{R}^d} \int_{\mathbb{R}^d} f(z - y) g(z) \, \mathrm{d}z \, \mathrm{d}y \\
& \leq 3^d \| f\|_{L^1} \|g \|_{L^1}.
\end{split}
\end{equation*} \par
Now that we have established (\ref{auxI}) for all integrable $f$ and $g$ supported in cubes with side length one, we proceed to the general case. For each $k$ and $m$ in $\mathbb{Z}^d$, set $f_k = \chi_{Q_k} f$, $g_m = \chi_{Q_m} g$.
 \par 
Given $x \in \mathbb{R}^d$, we claim that if $I^\theta(f_k,g_m)(x) \neq 0$, then each $m_i$ satisfies $k_i \leq m_i \leq k_i + 2$. Indeed, under the hypothesis $I^\theta(f_k,g_m)(x) \neq 0$, we have that $x + (\theta -1)y \in Q_k$, $x+\theta y \in Q_m$, and so the conditions below hold
\[k_i \leq x_i + (\theta -1)y_i \leq k_i + 1,\] and \[m_i \leq x_i + \theta y_i \leq m_i + 1, \]
for each $i = 1, \ldots, d$. Since $|y| \leq 1$, the conditions above imply that
\[k_i \leq x_i + \theta y_i - y_i \leq x_i + \theta y_i + 1 \leq m_i + 2, \] and  \[m_i \leq x_i + \theta y_i = x_i + (\theta -1)y_i + y_i \leq k_i + 2,\]
which establishes the claim. So, for any fixed $k \in \mathbb{Z}^d$, if $I^\theta(f_k,g_m)(x) \neq 0$, then $m = k + l,$ where $l \in [-2,2]^d \cap \mathbb{Z}^d = F$. Note that $F$ contains at most $5^d$ elements. \par We have:
\begin{equation*}
|I^\theta(f,g)|^{\frac{1}{2}} \leq \sum_{k \in \mathbb{Z}^d} \sum_{m \in \mathbb{Z}^d} |I^\theta(f_k,g_m)|^{\frac{1}{2}} = \sum_{l \in F} \sum_{k \in \mathbb{Z}^d} |I^\theta(f_k,g_{k+l})|^{\frac{1}{2}}
\end{equation*}
and so, using the fact that (\ref{auxI}) with $C= 3^d$ holds for the functions $f_k$ and $g_{k+l}$, it follows that
\begin{equation*}
\begin{split}
\|I^\theta(f,g) \|_{L^{\frac{1}{2}}} & \leq \left( \sum_{l \in F} \sum_{k \in \mathbb{Z}^d} \|I^\theta(f_k,g_{k+l})\|_{L^{1/2}}^{\frac{1}{2}} \right)^2 \\
& \leq 3^d \left( \sum_{l \in F} \sum_{k \in \mathbb{Z}^d} \|f_k \|_{L^1}^{\frac{1}{2}} \| g_{k+l} \|_{L^1}^{\frac{1}{2}} \right)^2.
\end{split}
\end{equation*} 
Finally, applying the Cauchy-Schwarz inequality to the last term above yields
\begin{equation*}
\begin{split}
\|I^\theta(f,g) \|_{L^{\frac{1}{2}}} & \leq 3^d \left( \sum_{l \in F} \Big(\sum_{k \in \mathbb{Z}^d} \|f_k \|_{L^1} \Big)^{\frac{1}{2}}  \Big(\sum_{k \in \mathbb{Z}^d} \|g_{k+l} \|_{L^1} \Big)^{\frac{1}{2}} \right)^2 \\
& \leq 3^d \left(\sum_{l \in F} \| f\|_{L^1}^{\frac{1}{2}} \|g \|_{L^1}^{\frac{1}{2}} \right)^2 \\
& \leq 3^d 5^{2d} \| f\|_{L^1} \| g\|_{L^1}
\end{split}
\end{equation*} 
which concludes the proof.
\end{proof}

\subsection{A dilated version of \texorpdfstring{$I^\theta$}{\unichar{"1D43C}\unichar{"1DBF}}} \label{section_HA_dilated}
~ \par 
In this section we consider a dilated version of $I^\theta$, denoted by $I_j^\theta$, for $j \in \mathbb{Z}$. This is defined as follows:
\begin{equation*} \label{Ithetaj}
I_j^\theta(f,g)(x) = \int_{|y| \leq 2^j} f(x + (\theta -1)y) \, g(x + \theta y) \, \mathrm{d}y.
\end{equation*}

\begin{lemma}
The operator $I_j^\theta$ maps $L^1 \times L^1$ to $L^1$ uniformly in $\theta$ and in $j$. Precisely:
\begin{equation} \label{auxIj1}
\| I_j^\theta(f,g) \|_{L^1} \leq \|f \|_{L^1}  \|g \|_{L^1}.
\end{equation}
\end{lemma}
\begin{proof}
Let $f,g \in L^1$. By Fubini's theorem and the change of variables $x + (\theta -1)y = z$ it holds that 
\[ \| I_j^\theta(f,g) \|_{L^1} = \int_{|y|\leq 2^j} \int_{\mathbb{R}^d} f(z)  g(z+y) \, \mathrm{d}z \, \mathrm{d}y.\]
Using Fubini's theorem once more, together with the change of variables $z+y = w$, it follows that
\[
 \| I_j^\theta(f,g) \|_{L^1}  = \int_{\mathbb{R}^d} f(z) \int_{|w-z|\leq 2^j}   g(w) \, \mathrm{d}w \mathrm{d}z \leq \|f \|_{L^1} \|g \|_{L^1}
\]
as desired.
\end{proof}
\begin{lemma}
The operator $I_j^\theta$ maps $L^1 \times L^1$ to $L^{\frac{1}{2}}$ uniformly in $\theta$. Precisely:
\begin{equation} \label{auxIj2}
\| I_j^\theta(f,g) \|_{L^{\frac{1}{2}}} \leq 2^{dj}3^d5^{2d}  \|f \|_{L^1}  \|g \|_{L^1}.
\end{equation}
\end{lemma}
\begin{proof}
This is a consequence of (\ref{auxI}) via a dilation argument which we include for convenience. \par We have:
\begin{align*}
\| I_j^\theta(f,g) \|_{L^{\frac{1}{2}}} & = \left(\int_{\mathbb{R}^d} |I_j^\theta(f,g)(2^jx)|^{\frac{1}{2}} 2^{dj} \, \mathrm{d}x  \right)^2 \\
& = 2^{2dj} \left(\int_{\mathbb{R}^d} \Big( \int_{|y| \leq 2^j} f(2^jx + (\theta - 1)y) g(2^jx + \theta y)   \, \mathrm{d}y \Big)^{\frac{1}{2}} \mathrm{d}x \right)^2 \\
& = 2^{2dj} \left(\int_{\mathbb{R}^d} \Big( \int_{|y| \leq 1} f(2^j(x + (\theta - 1)y)) g(2^j(x + \theta y)) 2^{dj}  \, \mathrm{d}y \Big)^{\frac{1}{2}} \mathrm{d}x \right)^2
\\
& = 2^{3dj} \|I^\theta(f_j, g_j) \|_{L^{\frac{1}{2}}}
\end{align*}
where $f_j(x) = f(2^jx)$ and $g_j(x) = g(2^jx)$. \par Using (\ref{auxI}), it follows that
\begin{align*}
\| I_j^\theta(f,g) \|_{L^{\frac{1}{2}}} & \leq 2^{3dj} 3^d5^{2d} \int_{\mathbb{R}^d} f_j(x) \, \mathrm{d}x \int_{\mathbb{R}^d} g_j(x) \, \mathrm{d}x \\
& = 2^{3dj} 3^d5^{2d} \int_{\mathbb{R}^d} f(x) 2^{-dj} \, \mathrm{d}x \int_{\mathbb{R}^d} g(x) 2^{-dj} \, \mathrm{d}x \\
& = 2^{dj}3^d5^{2d} \|f \|_{L^1}  \|g \|_{L^1}, 
\end{align*}
as desired.
\end{proof}

\begin{lemma}
There exists $c>0$, depending only on $d$, such that for all measurable sets $E,A,B \subseteq \mathbb{R}^d$ it holds:
\begin{empheq}[left={\displaystyle \left( \int_E | I^\theta_j(\chi_A,\chi_B)|^{\frac{1}{2}} \, \mathrm{d}x \right)^{2}\leq \, \empheqlbrace}]{align} 
     c  |A| |B|\min\{2^{dj}, |E| \}, \label{est1} \\ 
     c  |A| |E| \min\{2^{dj}, |B|\}, \label{est2} \\
     c |B| |E| \min\{2^{dj}, |A|\}, \label{est3}
  \end{empheq}
and
\begin{gather}
\int_E | I^\theta_j(\chi_A,\chi_B)| \, \mathrm{d}x    \leq   c \min\{2^{dj}|E|, |A||B|\}. \label{est4}
\end{gather}
\end{lemma}
\begin{proof}
First we prove estimate (\ref{est1}). Using (\ref{auxIj2}) with $f = \chi_A$ and $g=\chi_B$ we have that
\begin{align*}
\left( \int_E | I^\theta_j(\chi_A,\chi_B)|^{\frac{1}{2}} \, \mathrm{d}x \right)^{2} & \leq \left( \int_{\mathbb{R}^d} | I^\theta_j(\chi_A,\chi_B)|^{\frac{1}{2}} \, \mathrm{d}x \right)^{2} \\
& \leq 3^d5^{2d} 2^{dj} |A| |B|.
\end{align*}
On the other hand, by (\ref{auxIj1}) and the Cauchy-Schwarz inequality it holds:
\begin{align*}
\left( \int_E | I^\theta_j(\chi_A,\chi_B)|^{\frac{1}{2}} \, \mathrm{d}x \right)^{2} & = \left( \int_{\mathbb{R}^2} \chi_E | I^\theta_j(\chi_A,\chi_B)|^{\frac{1}{2}} \, \mathrm{d}x \right)^{2} \\
& \leq |E| |A| |B| \\
& \leq 3^d5^{2d} |E| |A| |B|.
\end{align*}
Estimate (\ref{est1}) follows from combining the two estimates above. \par
Next, we turn our attention to estimates (\ref{est2}) and (\ref{est3}). We only give a proof of the former due to their symmetrical nature. First, we use the Cauchy-Schwarz inequality as above to obtain
\[\left( \int_E | I^\theta_j(\chi_A,\chi_B)|^{\frac{1}{2}} \, \mathrm{d}x \right)^{2} \leq |E| \int_{\mathbb{R}^d} I_j^\theta(\chi_A, \chi_B) \, \mathrm{d}x.
\]
There are two ways to estimate the integral on the right-hand side of the previous inequality. One way is by $|A| |B|$ (using (\ref{auxIj1})), and the other is as follows (using that $\chi_B \leq 1$):
\begin{align*}
\int_{\mathbb{R}^d} I_j^\theta(\chi_A, \chi_B) \, \mathrm{d}x & \leq \int_{\mathbb{R}^d} \int_{|y| \leq 2^j} \chi_A(x + (\theta -1)y) \, \mathrm{d}y \, \mathrm{d}x \\
& =  \int_{|y| \leq 2^j} \int_{\mathbb{R}^d} \chi_A(x + (\theta -1)y) \, \mathrm{d}x \, \mathrm{d}y \\
& \leq \nu_d 2^{dj} |A|
\end{align*}
where $\nu_d$ denotes the measure of the unit ball in $\mathbb{R}^d$. This proves (\ref{est2}). \par 
In order to prove (\ref{est4}), we observe that 
\[I_j^\theta(\chi_A, \chi_B) \leq \nu_d 2^{dj} \]
from which it follows that
\[\int_E | I^\theta_j(\chi_A,\chi_B)| \, \mathrm{d}x \leq \nu_d 2^{dj} |E|. \]
The previous inequality together with 
\[\int_E | I^\theta_j(\chi_A,\chi_B)| \, \mathrm{d}x \leq \int_{\mathbb{R}^d} | I^\theta_j(\chi_A,\chi_B)| \, \mathrm{d}x  \leq |A||B|\]
yields the desired estimate.
\end{proof}
\subsection{Bilinear Marcinkiewicz interpolation} \label{section_HA_interpolation}
~\par
Recall the definition of weak Lebesgue spaces. For $0 < r < \infty$, the weak $L^r$ space, denoted by $L^{r,\infty}$, is the space of all measurable functions $f$ on $\mathbb{R}^d$ such that 
\begin{equation}
\| f\|_{L^{r,\infty}} \coloneqq \sup_{\lambda > 0} \lambda \big| \big\{ x\in \mathbb{R}^d:\,\, |f(x)|>\lambda \big\}\big|^{\frac{ 1}{r}} < \infty.
\end{equation}
The map $\| \cdot \|_{L^{r,\infty}}$ is a quasi-norm, and the following holds \cite{grafakos2014classical}:
\begin{equation} \label{weaknorm}
\|f \|_{L^{r, \infty}} \leq \sup_{0 < |E| < \infty} |E|^{-\frac{1}{s} + \frac{1}{r}} \left(\int_E |f|^s \, \mathrm{d}x \right)^{\frac{1}{s}}
\end{equation}
where $0 < s < r$ and the supremum is taken over measurable sets $E \subseteq \mathbb{R}^d$ with finite measure.

\begin{definition}
Let $0 < p,q,r < \infty$. A bilinear operator $U$ acting on measurable functions is said to be of restricted weak type $(p,q,r)$ (with constant $c>0$) if
\begin{equation} \label{restrictedwt}
\|U(\chi_A, \chi_B) \|_{L^{r,\infty}} \leq c \, |A |^{\frac{1}{p}} |B|^{\frac{1}{q}}
\end{equation}
for all measurable sets $A$ and $B$ with finite measure.
\end{definition}

The next proposition, a version of the multilinear Marcinkiewicz interpolation, is the main step towards establishing Theorem \ref{uniformestimate}. It yields strong-type bounds for bilinear operators, assuming only a finite set of restricted weak-type estimates. For a proof of the general case, see \cite{grafakos2012multilinear} or \cite[Theorem 7.2.2 and Corollary 7.2.4]{grafakos2014modern}.

\begin{proposition} \label{Minterpolation}
Let $0<p_i, q_i, r_i < \infty$ for $i = 1,2,3$. Suppose that the points 
\[
\Big( \frac{1}{p_1}, \frac{1}{q_1}\Big), \quad
\Big( \frac{1}{p_2}, \frac{1}{q_2}\Big), \quad
\Big( \frac{1}{p_3}, \frac{1}{q_3}\Big),
\]
do not lie on the same line in $\mathbb R^2$. 
For $0<\theta_1, \theta_2,\theta_3  <1$ satisfying $\theta_1+ \theta_2+\theta_3 =1$ consider the points $0<p,q,r<\infty$ such that
\[
\Big( \frac{1}{p }, \frac{1}{q },\frac{1}{r }\Big)=
\theta_1\Big( \frac{1}{p_1}, \frac{1}{q_1},\frac{1}{r_1}\Big)+
\theta_2\Big( \frac{1}{p_2}, \frac{1}{q_2},\frac{1}{r_2}\Big)+
\theta_3\Big( \frac{1}{p_3}, \frac{1}{q_3},\frac{1}{r_3}\Big)
\]
and 
\[
\frac{1}{r} \le \frac{1}{r_1}+\frac{1}{r_2}+\frac{1}{r_3} %+ \frac{1}{r_4}.
\]
Let $U$ be a bilinear operator that is of restricted weak-type $(p_i,q_i,r_i)$ (with constant $c_i>0$) for all $i=1,2,3 $.   Then there is a constant $C >0$ 
depending only on   $p_i$, $q_i$, $r_i$, and $\theta_i$ ($i=1,2,3 $) such that 
\[
\big\| U(f, g)\big\|_{L^r} \le C \, c_1^{\theta_1} c_2^{\theta_2} c_3^{\theta_3} %c_4^{\theta_4}
\| f\|_{L^p} \| g\|_{L^q} 
\]
for all functions $f\in L^p$ and $g \in L^q$.
 \end{proposition}
 We note that the conclusion of Proposition~\ref{Minterpolation} is also valid in the interior of the convex hull of four (or more) points at which  initial restricted weak-type estimates are known. The reason is that any polygon can be written as a union of triangles.

\subsection{Proof of Theorem \ref{uniformestimate}} \label{section_HA_proofthm}
~ \par
Turning our attention to the bilinear fractional integral operator $I_\alpha^\theta$ defined by (\ref{integraloperator}), we note that by a dilation argument, if it maps $L^p \times L^q$ to $L^r$, then necessarily
\begin{equation} \label{thmcondition}
\frac{1}{p} + \frac{1}{q} = \frac{1}{r} + \frac{\alpha}{d}. 
\end{equation}
\par 
Moreover, Theorem \ref{uniformestimate} affirms that $I_\alpha^\theta$ is bounded uniformly in $\theta$ from $L^p \times L^q$ to $L^r$ when $(p,q)$ lies in the open square with vertices $(1,1)$, $(1, \frac{d}{\alpha})$, $(\frac{d}{\alpha}, 1)$, $(\frac{d}{\alpha}, \frac{d}{\alpha})$ and (\ref{thmcondition}) holds. In this case, we have:
\begin{enumerate}[(i)]
\item If $(p,q) = (1,1)$, then $r = \frac{d}{2d-\alpha}$,
\item If $(p,q) = (1, \frac{d}{\alpha})$, then $r = 1$,
\item If $(p,q) = (\frac{d}{\alpha}, 1)$, then $r = 1$,
\item If $(p,q) = (\frac{d}{\alpha}, \frac{d}{\alpha})$, then $r = \frac{d}{\alpha}$.
\end{enumerate}
\par
Set \[ (p_1, q_1, r_1) = \left(1,1, \frac{d}{2d-\alpha} \right), \quad (p_2, q_2, r_2) = \left(1,\frac{d}{\alpha}, 1\right)\]  \[(p_3, q_3, r_3) = \left( \frac{d}{\alpha}, 1, 1\right), \quad (p_4, q_4, r_4) = \left(\frac{d}{\alpha}, \frac{d}{\alpha}, \frac{d}{\alpha}\right).\] To establish Theorem \ref{uniformestimate} it suffices to prove that $I_\alpha^\theta$ is of restricted weak type $(p_i, q_i, r_i)$ (with constant $c_i$ that is independent of $\theta$), for $i = 1,2,3,4 $. Then, the result follows by bilinear Marcinkiewicz interpolation, Proposition \ref{Minterpolation}. \par That is, we need to prove the following estimates:
\begin{align}
& \| I_\alpha^\theta(\chi_A, \chi_B) \|_{L^{\frac{d}{2d-\alpha}, \infty}} \leq c_1 |A| |B|,  \label{rest1} \\  
& \| I_\alpha^\theta(\chi_A, \chi_B) \|_{L^{1, \infty}} \leq c_2 |A| |B|^{\frac{\alpha}{d}}, \label{rest2}\\     
&\| I_\alpha^\theta(\chi_A, \chi_B) \|_{L^{1, \infty}} \leq c_3 |A|^{\frac{\alpha}{d}} |B|, \label{rest3} \\   
& \| I_\alpha^\theta(\chi_A, \chi_B) \|_{L^{\frac{d}{\alpha}, \infty}} \leq c_4 |A|^{\frac{\alpha}{d}} |B|^{\frac{\alpha}{d}}, \label{rest4}
\end{align}
uniformly in $\theta$, for all measurable sets $A$ and $B$ with finite measure.
\par
In the proof of \eqref{rest1}-\eqref{rest4} we utilize the following lemma.
\begin{lemma}
Let $d \in \mathbb{N}$, $0 < \alpha < d$ and $a,b > 0$. There exists $c = c(d,\alpha) > 0$ such that
\begin{equation} \label{A1}
\left(\sum_{j \in \mathbb{Z}} 2^{\frac{(\alpha - d)j}{2}} (\min\{2^{dj}, a \})^{\frac{1}{2}} \right)^2 \leq c\, a^{\frac{\alpha}{d}},
\end{equation}
\begin{equation} \label{A2}
 \sum_{j \in \mathbb{Z}} 2^{(\alpha - d)j} \min\{2^{dj}a, b \}  \leq c\, a \left(\frac{b}{a}\right)^{\frac{\alpha}{d}}.
 \end{equation}
\end{lemma}
\begin{proof}
We only prove (\ref{A1}) as the other one is similar. Let $m = \max \{j \in \mathbb{Z} \ | \ 2^{dj} < a \}.$ Then
\begin{equation*}
\begin{split}
\sum_{j \in \mathbb{Z}} 2^{\frac{(\alpha - d)j}{2}} (\min\{2^{dj}, a \})^{\frac{1}{2}} & = \sum_{j = -\infty}^m 2^{\frac{\alpha j }{2}}+\sum_{j = m+1}^{\infty} 2^{\frac{(\alpha - d)j}{2}} a^{\frac{1}{2}} \\
& = \sum_{k = 0}^{\infty} 2^{-\frac{\alpha(k-m)}{2}} + \sum_{i = 0}^{\infty} 2^{\frac{(\alpha - d)}{2}(i + m + 1)} a^{\frac{1}{2}} \\
& = \Big(\sum_{k = 0}^{\infty} 2^{-\frac{\alpha k}{2}} \Big) 2^{\frac{\alpha m}{2}} + \Big(\sum_{i = 0}^{\infty}2^{\frac{(\alpha - d)i}{2}} \Big) 2^{\frac{(\alpha - d)}{2}(m + 1)} a^{\frac{1}{2}}.
\end{split}
\end{equation*}
The desired inequality is achieved by noting that $2^{\frac{\alpha m}{2}} < a^{\frac{\alpha}{2d}}$ and $2^{\frac{(\alpha - d)}{2}(m + 1)} \leq a^{\frac{\alpha - d}{2d}}$.
\end{proof}

\begin{proof}[Proof of Theorem \ref{uniformestimate}]
~\par
First, note that $\mathbb{R}^d$ can be expressed as the union of annuli: \[\mathbb{R}^d = \bigcup_{j \in \mathbb{Z}} \left(B(2^j) \setminus B(2^{j-1})\right),\] where $B(R)$ denotes the open ball in $\mathbb{R}^d$ centered at the origin with radius $R$. \par 
Therefore:
\begin{align*}
I_\alpha^\theta(f,g)(x)& \leq \sum_{j \in \mathbb{Z}} \int_{2^{j-1} \leq |y| \leq 2^j} f(x + (\theta - 1)y) g(x + \theta y) |y|^{\alpha - d} \, \mathrm{d}y \\ 
& \leq \sum_{j \in \mathbb{Z}} 2^{d -\alpha} \int_{2^{j-1} \leq |y| \leq 2^j} f(x + (\theta - 1)y) g(x + \theta y) 2^{(\alpha - d)j} \, \mathrm{d}y \\
& \leq 2^{d -\alpha} \sum_{j \in \mathbb{Z}} 2^{(\alpha - d)j} I_j^\theta(f,g)(x).
\end{align*}
\par
Let $A,B$ be measurable sets of $\mathbb{R}^d$ with finite measure. In what follows, the positive constant $C$ might change from line to line, but it will always be independent of $\theta$. 
\par To prove the restricted estimate (\ref{rest1}) we use (\ref{est1}), (\ref{weaknorm}) with $s = 1/2$, and (\ref{A1}) as follows:
\begin{align*}
\| I_\alpha^\theta(\chi_A, \chi_B)\|_{L^{\frac{d}{2d-\alpha},\infty}} & \leq C \sup_{0 < |E| < \infty} |E|^{-2 + \frac{2d-\alpha}{d}} \left(\int_E \Big|\sum_{j \in \mathbb{Z}} 2^{(\alpha -d)j} I_j^\theta(\chi_A, \chi_B) \Big|^{\frac{1}{2}} \, \mathrm{d}x \right)^2 \\
& \leq C \sup_{0 < |E| < \infty} |E|^{-\frac{\alpha}{d}} \left(\sum_{j \in \mathbb{Z}} 2^{\frac{(\alpha -d)j}{2}} \int_E I_j^\theta(\chi_A, \chi_B)^{\frac{1}{2}} \, \mathrm{d}x\right)^2 \\
& \leq C \, |A| |B| \sup_{0 < |E| < \infty} |E|^{-\frac{\alpha}{d}} \left(\sum_{j \in \mathbb{Z}} 2^{\frac{(\alpha - d)j}{2}} (\min\{2^{dj}, |E| \})^{\frac{1}{2}} \right)^2 \\
& \leq C \, |A| |B|. 
\end{align*}
\par
Next we prove (\ref{rest2}). Here we use (\ref{est2}), (\ref{weaknorm}) with $s = 1/2$, and (\ref{A1}) as follows:
\begin{align*}
\| I_\alpha^\theta(\chi_A, \chi_B)\|_{L^{1,\infty}} & \leq C \sup_{0 < |E| < \infty} |E|^{-2 + 1} \left(\int_E \Big|\sum_{j \in \mathbb{Z}} 2^{(\alpha -d)j} I_j^\theta(\chi_A, \chi_B) \Big|^{\frac{1}{2}} \, \mathrm{d}x \right)^2 \\
& \leq C \sup_{0 < |E| < \infty} |E|^{-1} \left(\sum_{j \in \mathbb{Z}} 2^{\frac{(\alpha -d)j}{2}} \int_E I_j^\theta(\chi_A, \chi_B)^{\frac{1}{2}} \, \mathrm{d}x\right)^2 \\
& \leq C \, \sup_{0 < |E| < \infty} |E|^{-1} |A| |E| \left(\sum_{j \in \mathbb{Z}} 2^{\frac{(\alpha - d)j}{2}} (\min\{2^{dj}, |B| \})^{\frac{1}{2}} \right)^2 \\
& \leq C \, |A| |B|^{\frac{\alpha}{d}}. 
\end{align*}
\par 
The estimate (\ref{rest3}) is based on (\ref{est3}) and is deduced similarly as the one above. \par 
Finally, we turn to (\ref{rest4}). Here we use (\ref{est4}), (\ref{weaknorm}) with $s = 1$, and (\ref{A2}) as follows:
\begin{align*}
\| I_\alpha^\theta(\chi_A, \chi_B)\|_{L^{\frac{d}{\alpha},\infty}} & \leq C \sup_{0 < |E| < \infty} |E|^{-1 + \frac{\alpha}{d}} \int_E \Big|\sum_{j \in \mathbb{Z}} 2^{(\alpha -d)j} I_j^\theta(\chi_A, \chi_B) \Big| \, \mathrm{d}x \\
& \leq C \sup_{0 < |E| < \infty} |E|^{-1 + \frac{\alpha}{d}} \sum_{j \in \mathbb{Z}} 2^{(\alpha -d)j} \int_E I_j^\theta(\chi_A, \chi_B) \, \mathrm{d}x \\
& \leq C \, \sup_{0 < |E| < \infty} |E|^{-1+\frac{\alpha}{d}}  \sum_{j \in \mathbb{Z}} 2^{(\alpha - d)j} \min\{2^{dj} |E|, |A||B| \}\\
& \leq C \, \sup_{0 < |E| < \infty} |E|^{-1+\frac{\alpha}{d}} |E| \left(\frac{|A| |B|}{|E|} \right)^{\frac{\alpha}{d}} \\
& = C \, |A|^{\frac{\alpha}{d}} |B|^{\frac{\alpha}{d}}. 
\end{align*}
This completes the proof of the theorem.
\end{proof}

\subsection{Proof of Proposition \ref{THLS}} \label{section_HA_proofcor}
~ \par 
First, we observe that the tensor $J_\alpha(f,g)$ is pointwise bounded by $ \int_0^1 I_\alpha^\theta (f,g) \, \mathrm{d}\theta$. Indeed, for any $x \in \mathbb{R}^d$ and $v(x) \in \mathbb{R}^d \setminus \{ 0 \}$ it holds that
\begin{equation*}
\begin{split}
|J_\alpha(f,g)(x) v(x)| & = \left| \int_0^1 \int_{\mathbb{R}^d} f(x + (\theta - 1)y) \, g(x + \theta y) |y|^{\alpha -d - 2} (y \cdot v(x)) y \, \mathrm{d}y \, \mathrm{d}\theta \right| \\
& \leq \int_0^1 \int_{\mathbb{R}^d} f(x + (\theta - 1)y) \, g(x + \theta y) |y|^{\alpha -d - 2} |y \cdot v(x)| |y| \, \mathrm{d}y \, \mathrm{d}\theta \\
& \leq  \int_0^1 I_\alpha^\theta(f,g)(x) \, \mathrm{d}\theta |v(x)|
\end{split}
\end{equation*}
and hence
\[|J_\alpha(f,g)(x)| = \sup_{v(x) \neq 0} \frac{|J_\alpha(f,g)(x) v(x)|}{|v(x)|} \leq \int_0^1 I_\alpha^\theta(f,g)(x) \, \mathrm{d}\theta.\] \par
Therefore, using Jensen's inequality we deduce that 
\begin{equation*}
\|J_\alpha(f,g) \|_{L^r}^r  \leq  \int_{\mathbb{R}^d} \left|  \int_0^1 I_\alpha^\theta(f,g) \, \mathrm{d}\theta \right|^r \, \mathrm{d}x \leq  \int_{\mathbb{R}^d} \int_0^1 |I_\alpha^\theta(f,g)|^r \, \mathrm{d}\theta \, \mathrm{d}x.
\end{equation*}
Now, we use Fubini's theorem to obtain
\begin{equation*}
\|J_\alpha(f,g) \|_{L^r}^r  \leq    \int_0^1 \int_{\mathbb{R}^d} |I_\alpha^\theta(f,g)|^r \,  \mathrm{d}x \, \mathrm{d}\theta \\
 =  \int_0^1 \| I_\alpha^\theta(f,g)\|_{L^r}^r \, \mathrm{d}\theta
\end{equation*}
from which the desired result follows upon applying Theorem \ref{uniformestimate}.

\section{Stability for Euler-Riesz systems} \label{section_ER}
In this section, we establish a stability result for smooth solutions of an Euler-Riesz system with periodic boundary conditions, written according to identity (\ref{divSalpha}).
Two smooth solutions are compared using the relative energy functional. Using the abstract formalism developed in \cite{giesselmann2017relative}, we derive an identity that describes the time evolution of the relative energy. The right-hand side of the relative energy identity is controlled
with the help of the HLS inequality and then Gronwall's lemma provides a stability result. \par
Stability results of this type have been obtained for similar systems of equations, where one of the considered solutions is assumed to be merely a weak or even measure-valued solution, yielding a weak-strong uniqueness or measure-valued versus strong uniqueness principle (see \cite{alves2024role, alves2024weak, carrillo2024dissipative,lattanzio2017gas} and references therein).
The result obtained here can be phrased in the language of weak-strong stability, but we avoid doing that and we refer to 
\cite{alves2024weak} for details of such a formulation.
\par
Let $T > 0$ and denote by $\mathbb{T}^d$ the $d$-dimensional open cube $(-1/2, 1/2)^d$. Consider the following Euler-Riesz system in $(0,T) \times \mathbb{T}^d$, expressed using the abstract functional framework developed in \cite{giesselmann2017relative}:
\begin{equation} \label{ER1}
\begin{cases}
\partial_t \rho + \nabla \cdot (\rho u) = 0, \\
\partial_t (\rho u) + \nabla \cdot \big(\rho u \otimes u) + \rho \nabla \dfrac{\delta \mathcal{E}}{\delta \rho} (\rho)= 0,  \\
\rho|_{t=0} = \rho_0, \ u|_{t=0} = u_0,
\end{cases}
\end{equation}
with the potential energy functional $\mathcal{E}$ defined as 
\begin{equation*} \label{functionalE}
\mathcal{E}(\rho) = \int_{\mathbb{R}^d} h(\rho) + \kappa \tfrac{1}{2} \rho (K_\alpha \ast \rho) \, \mathrm{d}x
\end{equation*}
where $h(\rho) = \frac{1}{\gamma -1} \rho^\gamma$ is the internal energy function, $K_\alpha$ is the kernel given by (\ref{kernelK}),
and the constant $\kappa$ represents the interaction strength and for this section it is allowed to take positive and negative values.
The size of $|\kappa|$ will be restricted  to ensure that the relative energy is nonnegative. \par 
The density $\rho$ and velocity $u$ are assumed to be periodic in space with unit period.

\subsection{Relative energy identity} \label{section_ER_relative}
~\par 
The functional derivative $\delta \mathcal{E} / \delta \rho$ is given by
\begin{equation} \label{funct_deriv}
\dfrac{\delta \mathcal{E}}{\delta \rho}(\rho) = h^\prime(\rho) + \kappa K_\alpha \ast \rho
\end{equation}
which can be computed through the formula 
\[\left< \dfrac{\delta \mathcal{E}}{\delta \rho}(\rho) , \varphi\right>  = \int_{\mathbb{T}^d} \dfrac{\delta \mathcal{E}}{\delta \rho}(\rho) \varphi  \, \mathrm{d}x \coloneq \lim_{\delta \to 0} \frac{\mathcal{E}(\rho + \delta \varphi) - \mathcal{E}(\rho)}{\delta} \]
where $\varphi$ is an arbitrary test function. \par 
Furthermore, using (\ref{divSalpha}), we have  
\begin{equation}
\rho \nabla \dfrac{\delta \mathcal{E}}{\delta \rho}(\rho) = \nabla \cdot R_\alpha(\rho)
\end{equation}
where $R_\alpha(\rho) = p(\rho) I_d + \kappa S_\alpha(\rho)$, with $p(\rho) = \rho^\gamma$ being the pressure function. Note that the calculations in the Appendix \ref{appendix1} that lead to identity (\ref{divSalpha}) are valid if one replaces $\mathbb{R}^d$ by $\mathbb{T}^d$ due to the symmetrical assumption on the torus.\par 
The relative potential energy functional $\mathcal{E}( \cdot | \cdot)$ is defined as follows:
\begin{align*}
\mathcal{E}(\rho | \bar \rho)  & = \mathcal{E}(\rho) - \mathcal{E}(\bar \rho) - \left< \dfrac{\delta \mathcal{E}}{\delta \rho}(\bar \rho) , \rho - \bar \rho \right>  \\
& = \int_{\mathbb{T}^d} h(\rho | \bar \rho) + \kappa \tfrac{1}{2} (\rho - \bar \rho) \big(K_\alpha \ast (\rho - \bar \rho)\big) \, \mathrm{d}x
\end{align*}
where $h(\rho | \bar \rho) = h(\rho) - h(\bar \rho) - h^\prime(\bar \rho)(\rho - \bar \rho)$.
\par 
Next, we present the evolution of $\mathcal{E}(\rho | \bar \rho)$ over time, assuming that $\rho$ and $\bar \rho$ evolve according to system (\ref{ER1}). For the full details of the calculations involved, refer to \cite{giesselmann2017relative}. The following holds:
\begin{equation} \label{relP}
\frac{\mathrm{d}}{\mathrm{d}t} \mathcal{E}(\rho | \bar \rho) =  - \int_{\mathbb{T}^d} \nabla \bar u : R_\alpha(\rho | \bar \rho) \, \mathrm{d}x  - \left< \frac{\delta \mathcal{E}}{\delta \rho}(\rho) - \frac{\delta \mathcal{E}}{\delta \rho}(\bar \rho)  , \nabla \cdot \big( \rho(u - \bar u) \big) \right>
\end{equation}
where \[R_\alpha(\rho | \bar \rho) = p(\rho | \bar \rho) I_d + \kappa S_\alpha(\rho | \bar \rho). \]  \par 
Now, the linear velocities satisfy 
\begin{equation*}
\begin{dcases}
\partial_t u + (u \cdot \nabla ) u + \nabla \frac{\delta \mathcal{E}}{\delta \rho}(\rho) = 0, \\
\partial_t \bar u + (\bar u \cdot \nabla ) \bar u + \nabla \frac{\delta \mathcal{E}}{\delta \rho}(\bar \rho) = 0,
\end{dcases}
\end{equation*}
from which it can be deduced that
\begin{equation} \label{relK}
\begin{split}
\frac{\mathrm{d}}{\mathrm{d}t} \int_{\mathbb{T}^d} \tfrac{1}{2} \rho |u - \bar u|^2 \, \mathrm{d}x =  & - \int_{\mathbb{T}^d} \nabla \bar u : \rho (u - \bar u) \otimes (u - \bar u) \, \mathrm{d}x \\
&  + \left< \frac{\delta \mathcal{E}}{\delta \rho}(\rho) - \frac{\delta \mathcal{E}}{\delta \rho}(\bar \rho)  , \nabla \cdot \big( \rho(u - \bar u) \big) \right>.
\end{split}
\end{equation}
Combining (\ref{relP}) with (\ref{relK}) yields the relative total energy identity:
\begin{equation} \label{relTotal}
\frac{\mathrm{d}}{\mathrm{d}t} \left( \mathcal{E}(\rho | \bar \rho) + \int_{\mathbb{T}^d} \tfrac{1}{2} \rho |u - \bar u|^2 \, \mathrm{d}x \right)  =   - \int_{\mathbb{T}^d} \nabla \bar u : \big( \rho (u - \bar u) \otimes (u - \bar u) + R_\alpha(\rho | \bar \rho) \big) \, \mathrm{d}x.\\
\end{equation}

\subsection{Stability of smooth solutions} \label{section_ER_stability}
~\par 
Let $(\rho ,u)$ and $(\bar \rho, \bar u)$ be two smooth solutions of (\ref{ER1}), and suppose additionally that for $(\bar \rho, \bar u)$
the density $\bar \rho$ is bounded away from vacuum, that is, there exist $\bar \delta > 0$ and $\bar M < \infty$ such that 
\begin{equation*}
\bar \delta \leq \bar \rho(t,x) \leq \bar M \quad \text{for} \ (t,x) \in [0,T) \times \mathbb{T}^d
\end{equation*}
and also \[\nabla \bar u \in L^\infty\big(0,T;L^\infty(\mathbb{T}^d)\big).\]
\par 
The solutions $(\rho, u)$ and $(\bar \rho, \bar u)$ satisfy the relative energy identity (\ref{relTotal}). 
Let $\Psi: [0,T) \to \mathbb{R}$ denote the relative energy between this pair of solutions,
\begin{align*}
\Psi(t) &=  \int_{\mathbb{T}^d} \tfrac{1}{2} \rho |u - \bar u|^2 \, \mathrm{d}x + \mathcal{E}(\rho | \bar \rho)  \\
& = \int_{\mathbb{T}^d} \tfrac{1}{2} \rho |u - \bar u|^2 + h(\rho | \bar \rho)+\kappa \tfrac{1}{2} (\rho - \bar \rho) \big( K_\alpha \ast (\rho - \bar \rho) \big) \, \mathrm{d}x.
\end{align*}
The objective is to prove a stability estimate
connecting the behavior at time $T$ to the initial behavior at time zero.
 The  identity (\ref{relTotal}) yields
\begin{equation} \label{Psirel}
\begin{split}
\frac{\mathrm{d}}{\mathrm{d}t} \Psi(t)  = & - \int_{\mathbb{T}^d} \nabla \bar u : \rho (u - \bar u) \otimes (u - \bar u) \, \mathrm{d}x  - \int_{\mathbb{T}^d} (\nabla \cdot \bar u) p(\rho | \bar \rho) \, \mathrm{d}x \\
& - \kappa \int_{\mathbb{T}^d} \nabla \bar u : S_\alpha(\rho | \bar \rho) \, \mathrm{d}x.
\end{split}
\end{equation} 
\par 
To use $\Psi$ as a yardstick for comparing the two solutions, we need to show that $\Psi$ is nonnegative. This is based on two key ingredients. First, 
the HLS inequality (\ref{HLS}) gives
\[ \| (\rho - \bar \rho) \big(K_\alpha \ast (\rho - \bar \rho)\big) \|_{L^1(\mathbb{T}^d)} \leq C_0 \, \|\rho - \bar \rho \|_{L^p(\mathbb{T}^d)}^2 \]
for some positive constant $C_0 = C_0(\alpha, d)$, where $p = 2d/(d + \alpha)$. This is improved by using interpolation
and properties of the function $h(\rho  | \bar \rho)$ to show (see 
\cite[Lemma 3.6]{lattanzio2017gas} for the Newtonian potential and \cite[Proposition 4.2]{alves2024weak} for the general case):
\begin{lemma} \label{Cstar}
Consider the function $h(\rho) = \tfrac{1}{\gamma - 1}\rho^\gamma$ with $\gamma \geq 2 - \alpha/d$ and $0 < \alpha < d$. Let $\rho \in L^\gamma(\mathbb{T}^d)$ be nonnegative, and let $\bar \rho \in L^\infty(\mathbb{T}^d)$ be bounded away from vacuum. Then, there exists a positive constant $C_*$ such that 
\begin{equation}
 \| (\rho - \bar \rho) K_\alpha \ast (\rho - \bar \rho) \|_{L^1(\mathbb{T}^d)}  \leq C_*  \int_{\mathbb{T}^d} h(\rho | \bar \rho) \, \mathrm{d}x .
\end{equation}
\end{lemma}

Choosing $\kappa$ so that 
$0 < |\kappa| < \frac{2}{C_*}$  and setting $\lambda \coloneqq 1 - \frac{|\kappa| C_*}{2} > 0$, we obtain 
\[
\lambda \int_{\mathbb{T}^d} h(\rho | \bar \rho) \, \mathrm{d}x \leq \int_{\mathbb{T}^d} h(\rho | \bar \rho) + \kappa \tfrac{1}{2}(\rho - \bar \rho) K_\alpha * (\rho - \bar \rho) \, \mathrm{d}x 
 \]
from which the nonnegativity of $\Psi$ follows.
\par

Next, we bound the terms on the right-hand side of identity (\ref{Psirel}) in terms of $\Psi$. The first term is bounded by the relative kinetic energy, and hence by $\Psi$. The bound for the second term is also clear as $p(\rho | \bar \rho) = (\gamma - 1) h(\rho | \bar \rho)$. Regarding the last term, we first observe that due to the quadratic nature of $S_\alpha(\rho)$ one has
\[S_\alpha(\rho | \bar \rho) = S_\alpha(\rho - \bar \rho).\]
Moreover, for any fixed time $t \in [0,T)$, the $L^1$-norm of $S_\alpha(\rho - \bar \rho)$ is bounded by
\begin{align*}
 \mathcal{I} \coloneqq \tfrac{1}{2} \int_0^1 \int_{\mathbb{T}^d} \int_{\mathbb{T}^d} |(\rho - \bar \rho)(x+(\theta -1)y)| |(\rho - \bar \rho)(x+\theta y)|  |y|^{\alpha - d} \, \rd x \, \rd  y \, \rd \theta
\end{align*}
where the dependency on time is omitted for simplicity. We then estimate
\begin{align*}
\mathcal{I} & \leq  \tfrac{1}{2} \int_0^1 \int_{\mathbb{T}^d} \int_{|z| < 1} |(\rho - \bar \rho)(z)| |(\rho - \bar \rho)(z+y)|  |y|^{\alpha - d}\, \rd z \, \rd  y \, \rd \theta \\
& \leq  \tfrac{1}{2}  \int_{\mathbb{R}^d} \int_{\mathbb{R}^d} |(\rho - \bar \rho)(z)| \chi_{(-2,2)^d}(z) |(\rho - \bar \rho)(w)| \chi_{(-2,2)^d}(w) |z-w|^{\alpha - d}\, \rd z \, \rd  w \\
& \leq C(\alpha, d) \, \|(\rho - \bar \rho)\chi_{(-2,2)^d}  \|_{L^p(\mathbb{R}^d)}^2
\end{align*}
where $p =2d/(d+\alpha)$, by the HLS inequality. Finally, the periodicity in space of $\rho - \bar \rho$ implies that 
\[\|(\rho - \bar \rho)\chi_{(-2,2)^d}  \|_{L^p(\mathbb{R}^d)} = 4^d \|\rho - \bar \rho \|_{L^p(\mathbb{T}^d)}. \]
Hence, similarly to Lemma \ref{Cstar}, 
\begin{align*}
\int_{\mathbb{T}^d} \nabla \bar u : S_\alpha(\rho | \bar \rho) \, \mathrm{d}x & \leq \kappa \|\nabla \bar u \|_{\infty} \|S_\alpha(\rho - \bar \rho) \|_{L^1(\mathbb{T}^d)} \\
& \leq C \int_{\mathbb{T}^d} h(\rho | \bar \rho) \, \mathrm{d}x  \leq C \Psi.
\end{align*}

In summary, we have obtained the following inequality: 
\[\frac{\mathrm{d}}{\mathrm{d}t} \Psi \leq C \Psi. \]
By Gronwall's lemma, for each $t \in [0,T)$, it follows that $\Psi(t) \leq e^{CT} \Psi(0)$ which, together with the strict convexity of the internal energy function $h$, yields the desired stability result.
\par
A weak-strong uniqueness theorem is proved in \cite[Theorem 3.1]{alves2024weak} following the general approach outlined above. 
The method of proof differs in the treatment of the nonlocal term, achieved here via the use of the representation formula \eqref{divSalpha}. This provides an improvement in the range of parameters $\alpha$ achieving the full range $0 < \alpha < d$. By contrast, the range of $\gamma$ is still restricted by $\gamma \ge 2 - \alpha/d$.
\appendix
\section{} \label{appendix1}

Here we give a formal proof that for our  symmetric kernel $K_\alpha: \mathbb{R}^d \to \mathbb{R}$, with 
$K_\alpha(x) = \mathcal K_\alpha(|x|)$ it holds that 
\begin{equation} \label{appendixS}
f\nabla K_\alpha \ast f = \nabla \cdot S_\alpha(f)
\end{equation}
for any sufficiently smooth $f : \mathbb{R}^d \to \mathbb{R}$, where the tensor $S_\alpha(f)$ is given by 
\begin{equation} \label{appendixSdef}
 S_\alpha (f)(x) = -\frac{1}{2} \int_{\mathbb{R}^d} \int_0^1 \mathcal K_\alpha^\prime(|y|) \, \frac{1}{|y|} \, f(x+(\theta - 1)y) \, f(x+\theta y) \, y \otimes y \, \mathrm{d}\theta \, \mathrm{d}y.
\end{equation} \par 
Note that for  $\mathcal K_\alpha (|y|) = \tfrac{1}{d-\alpha} |y|^{\alpha -d}$ we have 
$\mathcal K_\alpha^\prime (|y|) = - |y|^{\alpha - d-1}$ and thus the corresponding tensor is positive semi-definite for nonnegative $f$. \par 
To prove \eqref{appendixS}, we first deduce that
\begin{equation} \label{appendixS2}
f \nabla K_\alpha \ast f (x) = - \frac{1}{2} \int_{\mathbb{R}^d} \mathcal K_\alpha^\prime(y) \frac{y}{|y|} \nabla_x  \cdot \int_0^1 y  f(x + (\theta -1)y)  f(x+\theta y)\, \mathrm{d}\theta \, \mathrm{d}y. 
\end{equation}
Using the symmetry of the convolution one has:
\begin{equation*}
\begin{split}
( f \nabla K_\alpha \ast f ) (x) & = f(x) \int_{\mathbb{R}^d} \mathcal K_\alpha^\prime(|y|) \frac{y}{|y|} f(x - y) \, \mathrm{d}y \\
& = - f(x) \int_{\mathbb{R}^d} \mathcal  K_\alpha^\prime(|y|) \frac{y}{|y|} f(x + y) \, \mathrm{d}y
\end{split}
\end{equation*}
where in the second equality we used the change of variables $y \to -y.$ \par 
Hence
\[f \nabla K_\alpha \ast f (x) = \frac{1}{2} \int_{\mathbb{R}^d} \mathcal  K_\alpha^\prime(|y|) 
\frac{y}{|y|} f(x) \big(f(x-y)-f(x+y) \big) \, \mathrm{d}y.\]
Now it is claimed that 
\[ f(x) \big(f(x-y)-f(x+y) \big) = - \nabla_x \cdot \int_0^1 y f(x + (\theta -1)y) f(x+\theta y)\, \mathrm{d}\theta, \]
from which identity (\ref{appendixS2}) follows. Indeed,
\begin{align*}
- f(x) & \big(f(x-y)-f(x+y)) 
= \int_0^1 \frac{\mathrm{d}}{\mathrm{d}\theta} \Big( f(x + (\theta - 1)y)f(x + \theta y)    \Big) \, \mathrm{d}\theta \\
& = \int_0^1 \big(\nabla f(x + (\theta - 1)y) \cdot y \big) f(x+\theta y) + \big(\nabla f(x + \theta y) \cdot y \big) f(x + ( \theta-1)y) \, \mathrm{d}\theta \\
%& = \int_0^1 \big(\nabla f(x + (\theta - 1)y) \big) f(x+\theta y) + \big(\nabla f(x + \theta y) \big) f(x + (1-\theta)y) \, \mathrm{d}\theta \cdot y \\
& = \int_0^1 \nabla \big( f(x+(\theta - 1)y) f(x+\theta y) \big) \, \mathrm{d}\theta \cdot y \\
& = \nabla_x \cdot \int_0^1 y f(x + (\theta -1)y) f(x+\theta y)\, \mathrm{d}\theta
\end{align*}
as desired. \par
Consequently, component-wise one has:
\begin{equation*}
\begin{split}
\big(\nabla \cdot  S_\alpha(f)(x) \big)_i &= - \frac{1}{2} \nabla_x \cdot \int_{\mathbb{R}^d} \int_0^1 \mathcal K_\alpha^\prime(|y|) \frac{y_i}{|y|} y f(x+(\theta - 1)y) f(x+\theta y) \, \mathrm{d}\theta \, \mathrm{d}y \\
& = - \frac{1}{2} \int_{\mathbb{R}^d} \mathcal  K_\alpha^\prime (  |y| ) \frac{y_i}{|y|} \nabla_x  \cdot \int_0^1 y  f(x + (\theta -1)y)  f(x+\theta y)\, \mathrm{d}\theta \, \mathrm{d}y \\
& = \big(f \nabla K_\alpha \ast f(x) \big)_i
\end{split}
\end{equation*}
which establishes identity (\ref{appendixS}).

It is noted that for a $d$-dimensional cube $[-a,a]^d$ centered at the origin, with $a > 0$, and periodic functions with period equal to $2a$, the formulas  \eqref{appendixS} and \eqref{appendixSdef} are still valid with the integrations
performed over $[-a,a]^d$.

\section*{Acknowledgments}
We thank Prof.  {\sc Dennis Serre} for explaining the application of compensated integrability to Euler-Poisson systems that served
as a starting point to this work. NJA would like to thank  Dr. {\sc Andr\'{e} Guerra} for bringing his attention to bilinear integral operators. This research was partially funded by the Austrian Science Fund (FWF), project number 10.55776/F65. The research of AET was supported by KAUST baseline funds.

\end{document}